\documentclass[11pt]{article}
\usepackage[a4paper,margin=28mm]{geometry}
\usepackage{amsmath,amssymb,amsthm}
\usepackage{booktabs}
\usepackage[colorlinks=true,linkcolor=blue,citecolor=blue,urlcolor=blue]{hyperref}
\usepackage{microtype}
\newtheorem{theorem}{Theorem}[section]
\newtheorem{lemma}[theorem]{Lemma}

\newcommand{\C}{\mathbb C}
\newcommand{\R}{\mathbb R}
\newcommand{\Q}{\mathbb Q}
\newcommand{\D}{\mathcal D}
\DeclareMathOperator{\Res}{Res}
\title{On the Five-Point Case of Mordell's Problem}
\author{Tiancheng Xu \quad Guowei Dou \quad Wensheng Yu\\[4pt]
\small School of Electronic Engineering\\
\small Beijing University of Posts and Telecommunications\\
\small Beijing 100876, China}
\date{}

\begin{document}
\maketitle
\begin{abstract}
For five complex points $z_1,\ldots,z_5$ satisfying $\sum_i|z_i|^2=5$,
Mordell's discriminant inequality states that
$\prod_{i<j}|z_i-z_j|^2\leq5^5$, with equality precisely for a unit regular
pentagon, up to rotation and permutation. We organize its proof around a
direct comparison of discriminant values. A global maximizer gives a centered
stationary configuration, whose coefficient equations reduce the generic
case to two quartic branches. On each branch, the absolute value of a cubic polynomial gives the original
discriminant objective after accounting for the normalization scale.
Exact rational interval bounds exclude both branches from the maximum,
and the exceptional coefficient cases leave the regular pentagon. The
argument determines the maximizing configuration without reconstructing all
stationary shapes. We have completed a full Lean formalization of the
five-point inequality and both directions of its equality condition,
including all reductions and finite checks used in this proof.
\end{abstract}


\section{Setup}
We fix the discriminant convention, state the five-point inequality, and
outline the reductions used in its proof.

\subsection{Convention}
For $n$ complex points define
\[
 N(z)=\sum_{i=1}^n|z_i|^2,
 \qquad \D_n(z)=\prod_{1\leq i<j\leq n}|z_i-z_j|^2.
\]
The constraint $N(z)=n$ fixes the total squared radius. Coincident points
give $\D_n(z)=0$. For $n=5$ we write $\D=\D_5$.

\subsection{Related work}
The discriminant problem belongs to the study of geometric mean differences
under moment constraints, including Mulholland's work on plane
distributions~\cite{Mulholland1958}. Mordell considered the conjecture that
$\D_n(z)\leq n^n$ when $N(z)=n$ and proved the three-point
case~\cite{Mordell1960}. Chalk gave counterexamples for
$n\geq6$~\cite{Chalk1960}. Kailiang Lin established the four-point case in
\emph{On Mordell's problem}, published in \emph{Math Media} in
2011~\cite{Lin2011}.

The stationary equations also occur in the theory of relative equilibria of
identical planar point vortices. Aref reviews the connection between these
equilibria and polynomial roots~\cite{Aref2007}, and Dirksen and Aref study
nearby relative-equilibrium configurations~\cite{DirksenAref2011}.
Faug\`ere and Svartz give a symbolic method for polynomial systems invariant
under the symmetric group and apply it to vortex
equilibria~\cite{FaugereSvartz2012}.

Complex polynomial optimization gives another formulation of the problem.
Wang and Magron include Mordell's inequality among the examples studied with
their real moment-HSOS hierarchy~\cite{WangMagron2025}.

Wang's exact proof, submitted on September 30, 2026~\cite{Wang2026},
establishes reflection symmetry, classifies the five collision-free
equilibrium shapes, and compares their discriminants; its resultant encoding,
mixed-moment identities, saturated coefficient eliminations, the dual
invariant making $B=b^2$ real, and quartic branches with their coefficient
graph polynomials are the algebraic inputs
to the argument below.

We use these coefficient conditions to compare discriminant values directly.
After normalizing a nonzero quintic coefficient, the generic case has a real
invariant $B=b^2$ on one of two quartic branches. On each branch we construct
a cubic polynomial $V_j$ with $\D(z)=|V_j(B)|$ and prove
$|V_j(B)|<3125$ at every real root of the branch polynomial. This excludes the
generic case from the maximum before any stationary shape is reconstructed.
The exceptional coefficient cases then identify the regular pentagon.

\subsection{Statement}
From now on, $n=5$ and $\D=\D_5$.
\begin{theorem}[Five-point inequality]\label{thm:main}
Let $z_1,\ldots,z_5\in\C$ satisfy $N(z)=5$. Then
\begin{equation}\label{eq:main}
 \D(z)\leq3125=5^5.
\end{equation}
Equality holds if and only if there are a unit complex number $u$, a primitive
fifth root of unity $\omega$, and a permutation $\sigma$ of
$\{0,1,2,3,4\}$ such that
\begin{equation}\label{eq:equality}
 z_{i+1}=u\omega^{\sigma(i)}\qquad(0\leq i\leq4).
\end{equation}
\end{theorem}

\subsection{Proof strategy}
A sharp extremal bound requires identifying the maximizing configurations.
We therefore use the coefficient equations as necessary conditions and
compare the objective on their branches, retaining the real invariant
$B=b^2$ and the normalization scale throughout the calculation.

The argument proceeds in five stages. First, compactness and the regular
pentagon give a positive global maximum; translation and scaling show that
it is centered. Second, differentiation of a homogeneous polynomial envelope
yields the stationary equations. Third, the resultant and mixed-moment
identities give the coefficient constraints and the two generic quartic
branches. Fourth, reduction modulo each quartic produces a cubic polynomial
for the discriminant value, and rational interval bounds exclude every real
root from attaining $3125$. Fifth, direct calculations in the exceptional
coefficient cases leave $X^5+d$, whose roots form a unit regular pentagon.

The generic comparison uses the identity that makes $B$ real, then passes
directly to the discriminant. It replaces the subsequent reflection-symmetric
parametrizations and the classification by one, three, or five real roots
with two scalar value bounds. The resulting certificates need to cover the
real zero sets of the quartics, rather than identify their physical shapes.
Section~\ref{sec:analytic} derives the coefficient conditions, and
Section~\ref{sec:branches} constructs the value bounds and determines the
equality case.

\section{Reductions}\label{sec:analytic}
This section connects the original norm constraint to the polynomial
conditions used by the finite certificates. We first derive stationarity
at a global maximizer and then pass to the coefficients of its root polynomial.

\subsection{Stationary reduction}
\subsubsection{Existence, positivity, and centering}
The set $S=\{z\in\C^5:N(z)=5\}$ is compact, and $\D$ is continuous. Hence
$\D$ has a global maximizer on $S$. A unit regular pentagon belongs to $S$ and
has discriminant $3125$. One way to see the latter identity is to use the
roots of $X^5-1$: the product of its derivatives at those roots is $5^5$.
Thus every global maximizer satisfies
\begin{equation}\label{eq:lower}
 \D(z)\geq3125>0.
\end{equation}
In particular its five coordinates are pairwise distinct.

Centering follows from a translation and scaling argument. Write
$m=\frac15\sum_i z_i$ and $w_i=z_i-m$. Then
\begin{equation}\label{eq:center}
 \D(w)=\D(z),\qquad
 N(w)=N(z)-5|m|^2.
\end{equation}
For a positive-discriminant configuration, $N(w)>0$. If $m\neq0$, then
$N(w)<5$, so multiplying $w$ by $r=\sqrt{5/N(w)}>1$ restores the original
constraint. Since there are ten pairs,
\[
 \D(rw)=r^{20}\D(w)>\D(z),
\]
contradicting global maximality. Thus a global maximizer has
$\sum_i z_i=0$.

\subsubsection{The force equation}
For a collision-free configuration define
\[
 F_i(z)=\sum_{j\neq i}\frac{1}{z_i-z_j}.
\]
For a complex direction $h\in\C^5$, differentiation along the real parameter
$s$ gives
\begin{align}
 \left.\frac{d}{ds}\D(z+sh)\right|_{s=0}
 &=2\D(z)\operatorname{Re}\sum_i h_iF_i(z),\label{eq:dD}\\
 \left.\frac{d}{ds}N(z+sh)\right|_{s=0}
 &=2\operatorname{Re}\sum_i h_i\overline{z_i}.
\end{align}
The degree-$20$ homogeneity of $\D$ yields the unconstrained envelope
\[
 \D(w)\leq\D(z)\left(\frac{N(w)}5\right)^{10}
 \qquad(w\in\C^5)
\]
at a global maximizer $z$. Taking the derivative of the difference of the two sides at $w=z$ shows that
\[
 \operatorname{Re}\sum_i h_i\bigl(F_i(z)-2\overline{z_i}\bigr)=0
 \qquad(h\in\C^5).
\]
Directions supported on one coordinate, with values $1$ and $\mathrm i$,
give
\begin{equation}\label{eq:force}
 F_i(z)=2\overline{z_i}\qquad(1\leq i\leq5).
\end{equation}
\subsubsection{Conjugate scaling}
Coefficient normalization requires a possibly nonreal scale. For $t\neq0$ set
\begin{equation}\label{eq:scale}
 x_i=tz_i,\qquad y_i=\frac{\overline{z_i}}t,
 \qquad \kappa=|t|^{-2}>0.
\end{equation}
Then $y_i=\kappa\overline{x_i}$, and \eqref{eq:force} and its conjugate imply
\begin{equation}\label{eq:pair}
 \sum_{j\neq i}\frac1{x_i-x_j}=2y_i,
 \qquad
 \sum_{j\neq i}\frac1{y_i-y_j}=2x_i.
\end{equation}
Both configurations are centered and collision-free, and
$N(x)=5|t|^2=5/\kappa$.

\subsection{Resultant reduction}\label{sec:resultant}
The next step replaces the point coordinates by four quintic coefficients.
The resultant carries the stationary equations into polynomial identities,
while a positive scale parameter retains the original norm constraint.

\subsubsection{The centered root polynomial}
Let
\begin{equation}\label{eq:quintic}
 p(X)=\prod_{i=1}^5(X-x_i)=X^5+aX^3+bX^2+cX+d.
\end{equation}
The missing $X^4$ term follows from centering. At a simple root,
\[
 \frac{p''(x_i)}{2p'(x_i)}
 =\sum_{j\neq i}\frac1{x_i-x_j}=2y_i,
 \qquad p''(x_i)=4y_ip'(x_i).
\]
Put
\[
 \delta=\prod_i p'(x_i).
\]
For degree five the sign in the derivative-product formula is positive, so
$\delta=\prod_{i<j}(x_i-x_j)^2\neq0$. This complex discriminant is distinct
from the nonnegative real objective $D=\D(z)$.

Define the auxiliary resultant
\begin{equation}\label{eq:resultant}
 H(Y,T)=\Res_X\bigl(p(X),(Y+TX)p'(X)-p''(X)\bigr).
\end{equation}
Evaluation at the roots of the monic polynomial $p$ proves
\begin{equation}\label{eq:product}
 H(Y,T)=\delta\prod_i(Y+Tx_i-4y_i).
\end{equation}
Subtracting $5Tp$ from the second input reduces its degree to four. Thus its
Sylvester determinant has size $9\times9$, and its expansion gives the
coefficient expressions used below.

\subsubsection{The zero and first coefficients in $T$}
Define polynomial coefficient expressions $r_3,r_2,r_1,r_0$ and
$s_2,s_1,s_0$ by
\begin{align}
 H(Y,0)&=\delta Y^5+r_3Y^3+r_2Y^2+r_1Y+r_0,\label{eq:r}\\
 \partial_T H(Y,0)&=20\delta Y^3+s_2Y^2+s_1Y+s_0.\label{eq:s}
\end{align}
These are rational polynomials in $a,b,c,d$.
The product expression \eqref{eq:product}, combined with the dual stationary
equations, gives mixed-moment identities
\begin{align*}
 \sum_i x_iy_i&=5,&\quad \sum_i x_iy_i^2&=0,\\
 4\sum_i x_iy_i^3&=7\sum_i y_i^2,&
 2\sum_i x_iy_i^4&=3\sum_i y_i^3.
\end{align*}
They follow by summing the force equations against powers of $y_i$ and pairing
the two orientations of each unordered pair. Differentiating
\eqref{eq:product} at $T=0$ and applying these identities proves
\begin{equation}\label{eq:three}
 s_2=0,\qquad s_1=6r_3,\qquad s_0=2r_2.
\end{equation}
For example, the $Y^3$ coefficient is $4\delta\sum_i x_iy_i=20\delta$,
and the $Y^2$ coefficient vanishes because $\sum_i x_iy_i^2=0$.
The remaining two coefficients are handled by the displayed cubic and quartic
mixed moments and Newton identities for the centered polynomial of the $y_i$.

\subsubsection{Conjugation and the original discriminant}
The root polynomial of $y_i=\kappa\overline{x_i}$ has coefficients
$\kappa^2\overline a,\kappa^3\overline b,
\kappa^4\overline c,\kappa^5\overline d$.
Taking coefficients of \eqref{eq:product} at $T=0$ therefore gives
\begin{align}
 r_3&=16\delta\kappa^2\overline a,&
 r_2&=64\delta\kappa^3\overline b,\label{eq:dual1}\\
 r_1&=256\delta\kappa^4\overline c,&
 r_0&=1024\delta\kappa^5\overline d.\label{eq:dual2}
\end{align}
Furthermore, \eqref{eq:scale} implies
\begin{equation}\label{eq:norm}
 \kappa^{20}|\delta|^2=D^2.
\end{equation}
Thus the coefficients of a global maximizer satisfy
$\kappa>0$, $\delta\neq0$, \eqref{eq:three}, and
\eqref{eq:dual1}--\eqref{eq:norm}.

\section{Coefficient branches}\label{sec:branches}
The coefficient conditions reduce the proof to finitely many branches.
We describe their exact certificates, establish the value bounds, and then
recover the inequality on the original sphere.
\subsection{Normalizing a nonzero coefficient}
Scaling the roots by $t$ sends $(a,b,c,d)$ to
$(t^2a,t^3b,t^4c,t^5d)$. Algebraic closedness of $\C$ provides $t\neq0$
with $t^2a=1$ if $a\neq0$, or with $t^3b=1$ if $a=0$, $b\neq0$.
The scale is accounted for through $\kappa$ and \eqref{eq:norm}. There are
three exhaustive cases: $a\neq0$; $a=0$, $b\neq0$; and $a=b=0$.
The first case is split again according to whether its normalized $b$ is zero.

\subsection{The generic branch: two quartics}
Assume $a=1$ and $b\neq0$. Set
\[
 B=b^2,\qquad k=d/b.
\]
The three equations \eqref{eq:three}, expressed in $B,c,k$, are combined with
the saturation equation
\begin{equation}\label{eq:saturation}
 qB\delta=1.
\end{equation}
Since $b\neq0$ and $\delta\neq0$, we can take $q=(B\delta)^{-1}$.
Polynomial elimination then gives the following conditions.
\begin{lemma}[Generic elimination]\label{lem:elimination}
For coefficient data with $a=1$ and $b\neq0$,
\begin{equation}\label{eq:quartics}
 P_1(B)P_2(B)=0,
\end{equation}
where
\begin{align*}
 P_1(B)&=125B^4+7900B^3-37040B^2+44608B+512,\\
 P_2(B)&=2500B^4-114000B^3+338925B^2-11000B+6912.
\end{align*}
On each branch $P_j(B)=0$, explicit rational polynomials
$C_j,K_j,\Delta_j,R_j,T_j$ satisfy
\begin{align}
 c=C_j(B),\quad k=K_j(B),\quad
 \delta=\Delta_j(B),\quad r_3=R_j(B),\quad r_2=bT_j(B),\label{eq:graph}\\
 \delta T_j(B)^2=R_j(B)^3.\label{eq:dualinv}
\end{align}
\end{lemma}
\begin{proof}[Certificate proof]
Polynomial ideal-membership identities for \eqref{eq:three} and
\eqref{eq:saturation} yield \eqref{eq:quartics} and the graph equations
\eqref{eq:graph}. Substitution into the resultant coefficients and reduction
modulo $P_j$ give \eqref{eq:dualinv}. The coefficient graph polynomials are
listed in Appendix~\ref{app:polynomials}.
\end{proof}

\subsection{Constructing the discriminant value polynomials}
The branch equations determine $\Delta_j$ and $R_j$ by substituting
$c=C_j(B)$ and $k=K_j(B)$ into the discriminant and the resultant coefficient
$r_3$, then taking remainders modulo $P_j$. We now use these remainders to
recover the scale and the objective, rather than the point coordinates.

In $\Q[B]/(P_j)$, compute $A_j$ by reducing
$R_j(16\Delta_j)^{-1}$, and compute $V_j$ by reducing $A_j^5\Delta_j$.
Both remainders have degree at most three, and polynomial division verifies
\[
 16\Delta_j A_j-R_j\equiv0\pmod{P_j},
 \qquad A_j^5\Delta_j-V_j\equiv0\pmod{P_j}.
\]
Thus, on the corresponding coefficient branch,
\begin{equation}\label{eq:ratios}
 16\delta A_j(B)=R_j(B),\qquad A_j(B)^5\delta=V_j(B).
\end{equation}
The explicit cubic polynomials are given in Appendix~\ref{app:polynomials}.
Combining these identities with conjugacy and the discriminant scale relation
gives the following value representation.

\begin{lemma}[Real parameter and discriminant value]\label{lem:physical}
On either generic branch, $B\in\R$ and
\begin{equation}\label{eq:value}
 A_j(B)=\kappa^2,\qquad D^2=V_j(B)^2.
\end{equation}
\end{lemma}
\begin{proof}
From \eqref{eq:graph}--\eqref{eq:ratios} we obtain
$\delta r_2^2=b^2r_3^3$. Substituting \eqref{eq:dual1} with $a=1$ gives
\[
 4096\delta^3\kappa^6(\overline b^{\,2}-b^2)=0.
\]
The scalar factor is nonzero, so $\overline{b^2}=b^2$, which is exactly
the assertion that $B$ is real. Next, $16\delta A_j=r_3=16\delta\kappa^2$ gives
$A_j=\kappa^2$. The last relation in \eqref{eq:ratios} becomes
$\kappa^{10}\delta=V_j(B)$. Taking squared moduli and using
\eqref{eq:norm} proves \eqref{eq:value}.
\end{proof}

\subsection{Bounds at every real root}
\begin{lemma}[Quartic value bounds]\label{lem:bounds}
For $j=1,2$ and every $B\in\R$ with $P_j(B)=0$,
\begin{equation}\label{eq:strict}
 -3125<V_j(B)<3125.
\end{equation}
\end{lemma}
\begin{proof}[Certificate proof]
Cauchy's polynomial root bound first places every real root of both quartics
inside $[-1000,1000]$. A finite binary subdivision of this interval is checked
using exact rational interval arithmetic. At every leaf the checker proves
one of two statements: the interval evaluation of $P_j$ excludes zero; or
the interval evaluation of $V_j$ lies strictly inside $(-3125,3125)$.
In the former case the leaf contains no root. In the latter case every root
in the leaf satisfies the desired value bound.

Interval containment and the fact that the two child intervals cover their
parent give the result by induction on the subdivision tree. The two
certificates have 55 and 67 nodes, respectively. Their value-bound leaves
and the inductive argument are given in Appendix~\ref{app:interval}.
The remaining leaves exclude zero from the polynomial range, so the
subdivision yields the bound on the entire real zero set without a root
count.
\end{proof}

Since $D\geq0$, Lemmas~\ref{lem:physical} and \ref{lem:bounds} imply
$D<3125$. Indeed $D^2=V_j(B)^2$ and $|V_j(B)|<3125$.
This contradicts \eqref{eq:lower}, and excludes the entire generic branch from
the set of global maximizers.

\subsection{Exceptional branches}
\begin{lemma}[Exceptional coefficient data]\label{lem:exceptional}
The following statements hold for the indicated physical coefficient data.
\begin{enumerate}
 \item If $a=1$ and $b=0$, then $c=3/20$, $d=0$, and $D<3125$.
 \item If $a=0$ and $b=1$, the coefficient data are inconsistent.
 \item If $a=b=0$, then $c=0$ or $d=0$. The subcase $d=0$ has $D<3125$.
\end{enumerate}
\end{lemma}
\begin{proof}[Certificate and coefficient proof]
For the first two cases, explicit saturated elimination identities imply,
respectively, $20c-3=0$, $d=0$, and the contradiction $1=0$.
The first case uses 25 certificate steps; the second uses 20.
With $a=1$, $b=0$, $c=3/20$, $d=0$, the equation for $r_3$ yields
$\kappa^2=25/4$. Substitution into \eqref{eq:norm} gives
\[
 D=\frac{84375}{1024}<3125.
\]
For $a=b=0$, the $s_2$ equation reduces to $50000cd^3=0$, so $c=0$ or
$d=0$. In the latter case $\delta=256c^5\neq0$. The $r_1$ equation gives
$\kappa^4|c|^2=625/256$. Consequently
\[
 D^2=65536\left(\frac{625}{256}\right)^5,
 \qquad D=\frac{9765625}{4096}<3125.
\]
\end{proof}

\subsection{The remaining branch and the equality condition}
If $a=b=0$ and $c=0$, then $p(X)=X^5+d$. Since $\delta\neq0$, we have
$d\neq0$. Apply this argument to the unscaled root polynomial of the original
maximizer. Its five distinct roots have the same nonzero fifth power.
Choose one root $u$; the ratios of the roots to $u$ run through all fifth roots
of unity. Thus the configuration is a rotated and scaled regular pentagon.
All five radii are equal, and the original constraint $N(z)=5$ forces
$|u|=1$, giving \eqref{eq:equality}.

\begin{proof}[Proof of Theorem~\ref{thm:main}]
Choose a global maximizer on the original sphere. It has positive
discriminant, is centered, and produces coefficient data.
If $a\neq0$, normalize to $a=1$. The branches $b\neq0$ and $b=0$ are
excluded by Lemmas~\ref{lem:bounds} and \ref{lem:exceptional}.
If $a=0$ but $b\neq0$, normalize to $b=1$ and use the inconsistent branch
of Lemma~\ref{lem:exceptional}. If $a=b=0$, that lemma leaves $c=0$ or
$d=0$, and $d=0$ again contradicts the lower bound. The remaining case is
a unit regular pentagon. Hence every global maximizer has discriminant
$3125$, proving the upper bound for every point of the sphere.

If an arbitrary configuration on the sphere has discriminant $3125$, the
upper bound makes it a global maximizer, so the same classification applies.
Conversely a configuration of the form \eqref{eq:equality} has discriminant
$3125$, since rotation and permutation preserve the objective and the unit
pentagon attains that value. This proves both directions of the equality
condition.
\end{proof}

\section{Conclusion}
We have completed a full Lean~4~\cite{Lean2021} formalization of the
five-point inequality in Theorem~\ref{thm:main} and both directions of its
equality condition, using mathlib~\cite{Mathlib2020}. The proof covers the
compactness and stationary reductions, the resultant and mixed-moment
identities, the necessary coefficient elimination and scale relations, the strict
quartic value bounds, and the recovery of the regular pentagon. Its
interval evaluation procedure is proved sound in the computational approach
to real bounds described by Melquiond~\cite{Melquiond2008}. The complete code
is provided at \url{https://github.com/dgjszy/mordell-five}.

\appendix
\section{Explicit polynomials for the generic value comparison}
\label{app:polynomials}
The graph polynomials in Lemma~\ref{lem:elimination} are
\begin{align*}
 C_1(B)&=\frac{99}{446}-\frac{3093}{3568}B
       +\frac{6015}{14272}B^2+\frac{375}{57088}B^3,\\
 K_1(B)&=\frac{109}{223}-\frac{255}{446}B
       +\frac{1203}{7136}B^2+\frac{75}{28544}B^3,\\
 C_2(B)&=\frac{48432}{147845}-\frac{65461}{88707}B
       +\frac{13885}{88707}B^2-\frac{250}{88707}B^3,\\
 K_2(B)&=\frac{224852}{739225}-\frac{776941}{1182760}B
       +\frac{6436}{29569}B^2-\frac{285}{59138}B^3.
\end{align*}
The scale polynomials and value polynomials are
\begin{align*}
 A_1(B)&=\frac{67977}{3568}-\frac{195295}{14272}B
       +\frac{129875}{57088}B^2+\frac{8375}{228352}B^3,\\
 A_2(B)&=\frac{80041741}{7569664}-\frac{12757497155}{545015808}B
       +\frac{234927875}{34063488}B^2-\frac{20405375}{136253952}B^3,\\
 V_1(B)&=\frac{25821909}{32768}-\frac{109276155}{131072}B
       +\frac{41181075}{524288}B^2+\frac{2838375}{2097152}B^3,\\
 V_2(B)&=\frac{725183548903125}{496085499904}
       +\frac{111393859432078125}{3968683999232}B\\
       &\hspace{1em}-\frac{2378884787578125}{248042749952}B^2
       +\frac{208919281640625}{992170999808}B^3.
\end{align*}
Reduction modulo $P_j$ gives
$16\Delta_j(B)A_j(B)=R_j(B)$ and
$A_j(B)^5\Delta_j(B)=V_j(B)$ at every root of $P_j$.

\section{The interval value certificates}\label{app:interval}
Let $I=[\ell,u]$ have rational endpoints, and let $\operatorname{range}_I(f)$
be the interval obtained by evaluating an expression for $f$ with outward
containment using exact rational arithmetic. A leaf is accepted when
\[
 0\notin\operatorname{range}_I(P_j)
 \quad\text{or}\quad
 \operatorname{range}_I(V_j)\subset(-3125,3125).
\]
A split at a rational number $s$ replaces $I$ by the two child intervals
$[\ell,\min(u,s)]$ and $[\max(\ell,s),u]$. Their union covers the parent,
and a split node is accepted when both children are accepted.

Soundness follows by induction. A real zero of $P_j$ lying in an accepted
leaf cannot satisfy the first alternative, by interval containment, and
therefore satisfies the value bound. At a split it lies in at least one
child, to which the induction hypothesis applies. The proof establishes
strict endpoint inequalities in the second alternative. Together with the
Cauchy bound for the entire real root set, this is a finite proof of
Lemma~\ref{lem:bounds}.

The two subdivisions have 28 and 34 leaves. Respectively, 25 and 31 leaves
exclude zero from the range of $P_j$. The remaining three leaves in each
subdivision are displayed in Table~\ref{tab:valueleaves}. Exact rational
Horner evaluation gives the listed integer enclosures for $V_j$; every
endpoint lies strictly between $-3125$ and $3125$.

\begin{table}[ht]
\centering
\renewcommand{\arraystretch}{1.6}
\begin{tabular}{@{}cll@{}}
\toprule
Branch & Interval $I$ & Enclosure of $\operatorname{range}_I(V_j)$ \\
\midrule
$P_1$ & $[\frac{-69375}{1024},\frac{-34625}{512}]$ & $[-3077,-2205]$ \\
$P_1$ & $[\frac{-125}{64},0]$ & $[788,2716]$ \\
$P_1$ & $[0,\frac{125}{32}]$ & $[-2469,789]$ \\
$P_2$ & $[\frac{6375}{2048},\frac{1625}{512}]$ & $[539,2398]$ \\
$P_2$ & $[\frac{1625}{512},\frac{3375}{1024}]$ & $[-2951,929]$ \\
$P_2$ & $[\frac{347375}{8192},\frac{2779125}{65536}]$ & $[1860,2636]$ \\
\bottomrule
\end{tabular}
\caption{Value-bound leaves in the two interval certificates. All remaining
leaves exclude roots by interval evaluation of $P_j$.}
\label{tab:valueleaves}
\end{table}

The same certificate principle applies to any rational polynomials $P,V$
and positive rational threshold $L$. Given an interval containing all real
zeros of $P$, an accepted subdivision whose leaves either exclude zero from
$P$ or enclose $V$ in $(-L,L)$ proves $|V(\beta)|<L$ at every real zero
$\beta$ of $P$. The present application uses $P=P_j$, $V=V_j$, and $L=3125$.

\bibliographystyle{unsrt}
\bibliography{references}

@article{Mulholland1958,
  author  = {Mulholland, H. P.},
  title   = {Inequalities between the geometric mean difference and the polar moments of a plane distribution},
  journal = {Journal of the London Mathematical Society},
  year    = {1958},
  volume  = {s1-33},
  number  = {3},
  pages   = {260--270},
  doi     = {10.1112/jlms/s1-33.3.260b}
}

@article{Mordell1960,
  author  = {Mordell, L. J.},
  title   = {On a discriminant inequality},
  journal = {Canadian Journal of Mathematics},
  year    = {1960},
  volume  = {12},
  pages   = {699--704},
  doi     = {10.4153/CJM-1960-063-5}
}

@article{Chalk1960,
  author  = {Chalk, J. H. H.},
  title   = {A note on a discriminant inequality},
  journal = {Canadian Mathematical Bulletin},
  year    = {1960},
  volume  = {3},
  number  = {1},
  pages   = {7--9},
  doi     = {10.4153/CMB-1960-002-x}
}

@article{Lin2011,
  author  = {Lin, Kailiang},
  title   = {On Mordell's problem},
  journal = {Math Media},
  year    = {2011},
  volume  = {35},
  number  = {2},
  pages   = {51--58},
  url     = {https://www.math.sinica.edu.tw/media/pdf/d352/35205.pdf},
  note    = {In Chinese}
}

@article{Aref2007,
  author  = {Aref, Hassan},
  title   = {Point vortex dynamics: a classical mathematics playground},
  journal = {Journal of Mathematical Physics},
  year    = {2007},
  volume  = {48},
  number  = {6},
  pages   = {065401},
  doi     = {10.1063/1.2425103}
}

@article{DirksenAref2011,
  author  = {Dirksen, Tobias and Aref, Hassan},
  title   = {Close pairs of relative equilibria for identical point vortices},
  journal = {Physics of Fluids},
  year    = {2011},
  volume  = {23},
  number  = {5},
  pages   = {051706},
  doi     = {10.1063/1.3590740}
}

@inproceedings{FaugereSvartz2012,
  author    = {Faug\`ere, Jean-Charles and Svartz, Jules},
  title     = {Solving polynomial systems globally invariant under an action of the symmetric group and application to the equilibria of {N} vortices in the plane},
  booktitle = {Proceedings of the 37th International Symposium on Symbolic and Algebraic Computation},
  year      = {2012},
  pages     = {170--178},
  publisher = {ACM},
  doi       = {10.1145/2442829.2442856}
}

@article{WangMagron2025,
  author  = {Wang, Jie and Magron, Victor},
  title   = {A real moment-{HSOS} hierarchy for complex polynomial optimization with real coefficients},
  journal = {Computational Optimization and Applications},
  year    = {2025},
  volume  = {90},
  number  = {1},
  pages   = {53--75},
  doi     = {10.1007/s10589-024-00617-0}
}

@inproceedings{Melquiond2008,
  author    = {Melquiond, Guillaume},
  title     = {Proving bounds on real-valued functions with computations},
  booktitle = {Automated Reasoning: 4th International Joint Conference, IJCAR 2008},
  year      = {2008},
  volume    = {5195},
  pages     = {2--17},
  publisher = {Springer},
  series    = {Lecture Notes in Computer Science},
  doi       = {10.1007/978-3-540-71070-7_2}
}

@inproceedings{Lean2021,
  author    = {de Moura, Leonardo and Ullrich, Sebastian},
  title     = {The {Lean 4} Theorem Prover and Programming Language},
  booktitle = {Automated Deduction---CADE 28},
  year      = {2021},
  pages     = {625--635},
  publisher = {Springer},
  doi       = {10.1007/978-3-030-79876-5_37}
}

@inproceedings{Mathlib2020,
  author    = {{The mathlib Community}},
  title     = {The {Lean} Mathematical Library},
  booktitle = {Proceedings of the 9th ACM SIGPLAN International Conference on Certified Programs and Proofs},
  year      = {2020},
  publisher = {ACM},
  doi       = {10.1145/3372885.3373824}
}

@misc{Wang2026,
  author       = {Wang, Jie},
  title        = {The Five-Point Case of Mordell's Discriminant Inequality},
  year         = {2026},
  howpublished = {arXiv:2610.00358v1},
  url          = {https://arxiv.org/abs/2610.00358},
  note         = {Submitted September 30, 2026}
}

\end{document}